\documentclass[a4paper,12pt]{article}

\usepackage{amsmath}
\usepackage{amsfonts}
\usepackage{amssymb}
\usepackage{amsthm}
\usepackage[cp1250]{inputenc}
\usepackage{setspace}
\usepackage[T1]{fontenc}
\usepackage{geometry}
\usepackage{fancyhdr}
\usepackage{graphicx}
\usepackage{color}
\usepackage{enumerate}
\newcommand{\N}{{\ensuremath{\mathbb N}}}

\newcommand{\E}{{\ensuremath{\mathbb E}}}
\newcommand{\Pro}{{\ensuremath{\mathbb P}}}

\newcommand{\R}{{\ensuremath{\mathbb R}}}

\newcommand{\lv}{\lVert}
\newcommand{\rv}{\rVert}
\newcommand{\pmom}{\left\lv X \right\rv_p}
\newcommand{\kk}{\frac{1}{k}}
\newcommand{\norma}{\lv (\ai) \rv_{p}}
\newcommand{\normb}{\lv (\ai) \rv'_{k,p}}
\newcommand{\ai}{a_{i_1,...,i_d}}
\newcommand{\nawa}{\sup \left\{\sum_\sumator a_{i_1,...,i_d} \prod_{r=1}^d \left(1+ v^{(r)}_{i_r} \right) \ | \ \left( v^{(r)}_i \right) \in B^{(r)}_{p} \right\}}
\newcommand{\nawaloc}{\sup \left\{\sum_\sumator a_{i_1,...,i_d} \prod_{r=1}^d \left(1+ v^{(r)}_{i_r} \right) \ | \ \left( v^{(r)}_i \right) \in B_{p} \right\}}

\newcommand{\nawb}{\sup \left\{\sum_\sumator a_{i_1,...,i_d} \prod_{r=1}^d \prod_{l=1}^k \left(1+  v^{(r)}_{i_r,l} \right) \ | \ \left(v^{(r)}_{i,l}\right) \in B^{(r)}_{k,p} \  r=1,...d \right\}}
\newcommand{\nawbc}{\sup \left\{\sum_{i_1,\ldots,i_d} a_{i_1,...,i_d} \prod_{r=1}^d \prod_{l=1}^k \left(1+  v^{(r)}_{i_r,l} \right) \ | \ \left(v^{(r)}_{i,l}\right) \in B^{(r)}_{k,p}, \  r\leq d \right\}}
\newcommand{\nawbb}{\sup \left\{\sum_{i_1,\ldots,i_d} a_{i_1,...,i_d} \prod_{r=1}^d \prod_{l=1}^k \left(1+  v^{(r)}_{i_r,l} \right) \ | \ \left( v^{(r)}_{i,l}\right) \in D^{(r)}_{k,p}, \ r=1,...d \right\}}
\newcommand{\X}{X_{i_1} \cdot ... \cdot X_{i_d}}
\newcommand{\Xn}{X^{(1)}_{i_1} \cdot ... \cdot X^{(d)}_{i_d}}

\newcommand{\pp}{\prod_{r=1}^d\prod_{l=1}^k}
\newcommand{\ki}{\prod_{l=1}^k(1+a_{i,l})}

\newcommand{\aaa}{a_{i,l}}
\newcommand{\sumator}{{1\leq i_1,\ldots, i_d \leq n}}

\newcommand{\eps}{\varepsilon}
\providecommand{\keywords}[1]{\textbf{\textit{Keywords: }} #1}

\providecommand{\klas}[1]{\textbf{\textit{AMS MSC 2010: }} #1}

\newtheorem{twie}{Theorem}
\newtheorem{twr}[subsection]{Theorem}
\newtheorem{lem}[subsection]{Lemma}

\newtheorem{fak}[subsection]{Fact}

\newtheorem{rem}[subsection]{Remark}

\begin{document}

\title{Two-sided moment estimates for a class of nonnegative chaoses}
\author{Rafa{\l{}} Meller}
\date{}
\maketitle

\begin{abstract}
We derive two-sided bounds for moments of random multilinear forms (random chaoses) with nonnegative coeficients 
generated by independent nonnegative random variables $X_i$ which satisfy the following condition on the growth of moments:
$\lv X_i \rv_{2p} \leq A \lv X_i \rv_p$ for any $i$ and $p\geq 1$. 
Estimates are deterministic and exact up to multiplicative constants which depend only on the order of chaos 
and the constant $A$ in the moment assumption. \\
 \\
\klas{60E15}

\end{abstract}

\section{Introduction}
In this paper we study  homogeneous 
tetrahedral chaoses of order $d$, i.e. random variables of the form
$$S=\sum_{1\leq i_1,\ldots, i_d \leq n} \ai \X, $$
where $X_1,\ldots,X_n$ are independent random variables and $(\ai)$ is a multiindexed symmetric array of real numbers such that  $\ \ai=0$ if $i_l=i_m$ for some $m\neq l$, $m,l\leq d$.

Chaoses of order $d=1$ are just sums of independent random variables, object quite well understood. R. Lata{\l{}}a \cite{4} derived two-sided bounds for $\lv \sum a_i X_i \rv_p$ under general assumptions that either $a_i,X_i$ are nonnegative or $X_i$ are symetric.
The case $d\geq 2$ is much less understood. There are papers presenting two-sided bounds for moments of $S$ in special cases when $(X_i)$ have normal distribution \cite{5}, have logarithmically concave tails  \cite{6} or logarithmically convex tails \cite{7}. 

The purpose of this note is to derive two-sided bounds for $\lv S \rv_p$ if coefficients $(\ai)$ are nonnegative and 
$(X_i)$ are independent, nonegative and satisfy the following moment condition for some $k\in \N$,
\begin{align} \label{MC}
\lv X_i \rv_{2p} \leq 2^k \lv X_i \rv_p  \ \ \mbox{for every }p \geq 1.
\end{align}
The main idea is that if a r.v.\ $X_i$ satisfy \eqref{MC} then it is comparable with a product of $k$ i.i.d.\ variables 
with logarithmically concave tails. 
This way the problem reduces to the result of Lata{\l{}}a and {\L{}}ochowski \cite{2} which gives two-sided bounds for moments of nonnegative chaoses generated by r.v's with logarithmically concave tails.

\section{Notation and main results}

We set $\lv Y \rv_p=(\E|Y|^p)^{1/p}$ for a real r.v. $Y$ and $p\geq 1$, $\log(x)=\log_2(x)$ and
$\ln$ stands for the natural logarithm. 
By $C,t_0$ (sometimes $C(k,d), t_0(k,d)$) we denote constants, that may depend on $k,d$ and may vary from line to line. 
We write $A \sim_{k,d} B$ if $A\cdot C(k,d) \geq B$ and $B \cdot C(k,d) \geq A$.

Let $\{X^{(1)}_{i}\},\ldots,\{X^{(d)}_{i}\}$ be independent r.v's. 
We set 
\[
N^{(r)}_i(t)=-\ln \Pro (X^{(r)}_i \geq t).
\]
We say that $X^{(r)}_i$ has logarithmically concave tails if 
the function $N^{(r)}_i$ is convex. 
We put
$$B^{(r)}_p = \left\{ v \in \R^n \ | \ \sum_{i=1}^n N^{(r)}_i(v_i) \leq p \right\}$$
and
$$\norma= \nawa.$$

We will show the following
\begin{twr}\label{th1} 
Let $(X^{(r)}_i)_{r\leq d, i \leq n}$ be independent non-negative random variables satisfying \eqref{MC}
and $\E X^{(r)}_i=1$. Then for any nonnegative coeficients $(\ai)_{i_1,\ldots,i_d \leq n}$  we have
$$\frac{1}{C(k,d)} \norma \leq  \left\lv \sum_\sumator \ai \Xn \right\rv_p    \leq  C(k,d) \norma .$$
\end{twr}

Theorem \ref{th1} in the same way as in the proof of Theorem 2.2 in \cite{2}
yields the following two-sided bounds for tails of random chaoses.

\begin{twr}\label{ogony}
Under the assumptions of Theorem \ref{th1} there exist constants $0<c(k,d),C(k,d)<\infty$ depending only on $d$ and $k$ such that for any $t\geq0$ we have
$$\Pro \left( \sum_\sumator \ai \Xn \geq C(k,d)\norma \right)\leq e^{-t}$$
and
$$\Pro \left( \sum_\sumator \ai \Xn \geq c(k,d) \norma\right) \geq \min(c(k,d),e^{-t}). $$
\end{twr}

Now we are ready to present two-sided bounds for undecoupled chaoses. We define in this case $N_i(t)=-\ln \Pro (X_i \geq t)$,
$$B_p = \left\{ v \in \R^n \ | \ \sum_{i=1}^n N_i(v_i) \leq p \right\}$$
and
$$\norma = \nawaloc. $$
 
\begin{twr}\label{main}
Let $(X_i)_{i \leq n}$ be nonnegative independent r.v's satisfying \eqref{MC} and $\E X_i=1$. 
Then for any symmetric array of nonnegative coeficients $(\ai)_{i_1,\ldots,i_d\leq n} $ such that 
\begin{equation}\label{loca}
\ai=0 \textrm{ if } i_l=i_m \textrm{ for some } m \neq l, \ m,l \leq d
\end{equation}
we have
$$\frac{1}{C(k,d)} \norma  \leq  \left\lv \sum_\sumator \ai \X \right\rv_p    \leq  C(k,d) \norma. $$
Moreover,
\begin{align*}
&\Pro \left( \sum_\sumator \ai \X \geq C(k,d)\norma \right)\leq e^{-t}, 
\\
&\Pro \left( \sum_\sumator \ai \X \geq c(k,d) \norma \right) \geq \min(c(k,d),e^{-t}).
\end{align*}
\end{twr}
\begin{proof}
Let $S'=\sum \ai X^{(1)}_{i_1} \cdot  \ldots \cdot X^{(d)}_{i_d} $ be the decoupled version of $S=\sum \ai \X$. 
By the results of de la Pe{\~n}a and  Montgomery-Smith \cite{8} (one may use also the result of Kwapie{\'n} \cite{1}) 
moments and tails of $S$ and $S'$ are comparable up to constants which depend only on $d$. 
Hence Theorem \ref{main} follows by Theorems \ref{th1} and \ref{ogony}.

\end{proof}

\section{Preliminiares}
In this section we study properties of nonnegative r.v's satisfying condition \eqref{MC}. We will assume normalization
$\E X=1$ and define $N(t)=-\ln \Pro(X\geq t)$.

\begin{lem}\label{lemma0}
There exists a constant  $ C=C(k)$ such that for any $ x\geq 1$, and $ t\geq 1$ we have $N(Ctx)\geq t^\kk N(x)$. One may take $C=8^{k+1}$.
\end{lem}

\begin{proof}
We need to show that  
\begin{equation}\label{loc1}
\Pro(X \geq Ctx) \leq \Pro( X\geq x)^{t^\kk} \textrm{ for } t\geq 1,\ x\geq 1.
\end{equation}

It is enough to prove the assertion for $x< \frac{\lv X \rv_\infty}{2}$ because for $x \geq \frac{\lv X \rv_\infty}{2}$, \eqref{loc1} is obvious for $C> 2$ . In that case $x=\frac{1}{2} \lv X \rv_q $ for some $q \geq 1$ (since $\lv X \rv_1=1$). From the Paley-Zygmunt inequality and \eqref{MC}
\begin{align}
\Pro\left(X \geq x\right) &=\Pro\left(X^q \geq \frac{1}{2^q} \E X^q\right) \geq \left(1-\frac{1}{2^q}\right)^2 \left(\frac{\lv X\rv_q}{\lv X \rv_{2q}} \right)^{2q} \nonumber \\
&\geq \left(1-\frac{1}{2^q} \right)^2 \frac{1}{2^{2kq}} \geq \left(\frac{1}{2^{k+1}} \right)^{2q}.  \label{loc1234}
\end{align}
  Let $A=\lceil \kk \log(t) \rceil$. By \eqref{MC} we get $\lv X \rv_{q2^A} \leq 2^{kA} \lv X \rv_q$ hence Chebyshev's inequality yields

$$\Pro (X \geq Ctx)=\Pro\left(X \geq \frac{Ct}{2} \lv X \rv_q\right)\leq \Pro \left(X \geq \frac{Ct}{2^{kA+1}} \lv X \rv_{q2^A} \right)\leq \left( \frac{2^{kA+1}}{Ct} \right)^{q2^A}.$$
We have $2^A \geq t^\kk$ and $2^{kA}\leq 2^{k \left(\kk \log (t) +1 \right)}=t2^k$ so if $C=8^{k+1}$ then
\begin{equation}\label{loc123}
\Pro (X \geq Ctx) \leq \left(\frac{1}{4^{k+1}} \right)^{qt^\kk}= \left(\frac{1}{2^{k+1}} \right)^{2qt^\kk}.
\end{equation}
The assertion follows by \eqref{loc1234} and \eqref{loc123}
\end{proof}

In fact one may reverse the statement of Lemma \ref{lemma0}.

\begin{rem}
Let X be a nonnegative r.v., $\E X=1$ and there exist constants $C,\beta >0$ 
such that $N(Ctx) \geq t^\beta N(x)$ for $t,x\geq 1$. Then there exists $\bar{K}=\bar{K}(C,\beta)$ such that
$$\lv X \rv_{2p} \leq \bar{K} \lv X \rv_p \ \ \mbox{ for }p\geq 1.$$
\end{rem}

\begin{proof}
In this proof $K$ means constant which may depend on $C$ and $\beta$ and vary from line to line.  Integration by parts yields
\begin{align}
\E \left| \frac{X}{2C} \right|^{2p} &= \int_0^\infty 2p t^{2p-1} e^{-N(2Ct)}dt\leq  \pmom^{2p}+\int_{\pmom}^\infty 2p t^{2p-1}e^{-N(2Ct)} dt \nonumber \\
&\leq \pmom^{2p} + \int_{\pmom}^\infty 2p t^{2p-1} e^{-N(2\pmom)(\frac{t}{\pmom})^\beta} dt. \label{locab}
\end{align}
Let $\alpha=N(2\pmom)^{\frac{1}{\beta}}/\pmom$, substituting $y=\alpha t$ into \eqref{locab} we get
$$ \E \left| \frac{X}{2C} \right|^{2p} \leq \pmom^{2p} +\frac{1}{\alpha^{2p}} \int_0^\infty 2py^{2p-1} e^{-y^\beta} dy.$$
Thus
\begin{align}
\lv X \rv_{2p} &\leq 2C \pmom + \frac{2C}{\alpha} \left(\int_0^\infty 2py^{2p-1} e^{-y^\beta} dy\right)^{1/2p}
\nonumber 
\\
&=2C \pmom + \frac{2C}{\alpha} \left(\frac{2p}{\beta} \Gamma(\frac{2p}{\beta})\right)^{1/2p} 
\leq 2C\pmom + K \frac{p^\frac{1}{\beta}}{\alpha}. \label{loca1}
\end{align}
By Chebyshev's inequality $N(2\pmom)  \geq p \ln 2$ and the assertion follows by \eqref{loca1}.
\end{proof}

Now we state the crucial technical lemma.

\begin{lem}\label{lemma1}
There exists $C,c,t_0$ which depend on $k$, a probability space with a version of $X$ and nonnegative i.i.d. r.v's $Y_{1},...,Y_{k}$ with the following properties
\begin{enumerate}[(i)]
\item $C(X+t_0) \geq Y_{1} \cdot \ldots \cdot Y_{k}$,\label{i}
\item $C( Y_{1}\cdot \ldots \cdot Y_{k} +t_0) \geq X$, \label{ii}
\item $ Y_{1},\ldots,Y_{k}$ have log-concave tails, \label{iii}
\item $H(t) \leq N(t^k)\leq H(C  t) \textrm{ for } t\geq t_0$, where $H(t)=-\ln \Pro \left(Y_{l}\geq t \right)$,\label{iv}
\item $\frac{1}{C} \leq \E Y_{l} \leq C$.  \label{v}
\end{enumerate}
\end{lem}

\begin{proof}
Let $M(t)=N(t^k)$. By Lemma \ref{lemma0} there exists $C$ (depending on $k$) such that $M(C \lambda t) \geq \lambda M(t)$ for all $\lambda\geq 1, \ t\geq 1$. By \cite[Lemma 3.5]{3} (applied with $t_0=1$) there exists convex nondecreasing function $H$, constants $C=C(k), t_0=t_0(k)>0$ such that 
\begin{align}
H(t) \leq M(t)\leq H(C  t) \textrm{ for } t\geq  t_0 \nonumber \\
H(t)=0 \textrm{ for }  t \leq t_0 \label{kluwlas}
\end{align}
Let $Y_i$ be nonnegative i.i.d, r.v's such that $\Pro (Y_{l} \geq t)=e^{-H(t)}$, then \eqref{iii} and \eqref{iv} hold.

Now we verify \eqref{i} and \eqref{ii}. For $t\geq \max\{1,t_0\}$ we have 
\begin{align}
\Pro\left(\prod_{l=1}^k Y_{l}\geq t \right) &\geq \prod_{l=1}^k \Pro \left(Y_{l} \geq t^\kk\right)=e^{-kH(t^\kk)}\geq e^{-k M(t^\kk)}=e^{-k N(t)}\nonumber \\
&\geq e^{-N(C k^k t)}=\Pro\left(X \geq Ck^k t \right), \label{locb}
\end{align}
where the last inequality comes from Lemma \ref{lemma0}. Furhermore, 
\begin{align} \label{r0}
\Pro\left(\prod_{l=1}^k Y_{l}\geq C^k t \right)\leq \sum_{l=1}^k \Pro(Y_{l} \geq Ct^\kk )= ke^{-H(Ct^\kk)}\leq ke^{-M(t^\kk)}=ke^{-N(t)}.
\end{align}

By Chebyshev's inequality $1=\E X \geq e \Pro (X \geq e) = e^{1-N(e)}$, so $N(e)\geq 1$ and by Lemma \ref{lemma0} we get for $t\geq 1, \ N(Cte)\geq t^\kk N(e) \geq t^\kk$. Thus 
\begin{equation} \label{r1}
\ln k-N(t)\leq -\frac{1}{2}N(t) \ \textrm{ for } t \geq eC\max(1,2\ln k)^k.
\end{equation}

Lemma \ref{lemma0} also gives $\frac{N(t)}{2} \geq N(\frac{t}{2^kC})$ for $t> 2^k C$, so from \eqref{r1} and \eqref{r0}
\begin{equation}\label{loc1111}
\Pro\left(\prod_{l=1}^k Y_{l}\geq C^k t \right) \leq  e^{-N(\frac{t}{2^kC})}=\Pro \left(X \geq \frac{t}{2^k C} \right).
\end{equation}

Inequalities \eqref{locb} and \eqref{loc1111} implies \eqref{i} and \eqref{ii}. To show \eqref{v} observe that
\begin{align*} 
\left( \E Y_{l} \right)^k=\E Y_{1} \cdot \ldots \cdot Y_{k} \leq C \left( \E X +t_0 \right)=C(1+t_0)
\end{align*}
an by \eqref{kluwlas}
\begin{align*}
\E Y_{l} \geq t_0>0
\end{align*}

\end{proof}

\section{Proof of Theorem} \label{dow}

Let $X_i^{(r)}$, $r\leq d,i\leq n$ satisfy the assumptions of Theorem \ref{th1}. 
By Lemma \ref{lemma1} we may assume (enlarging if necessary the proobability space) that there exist independent 
r.v's $Y^{(r)}_{i,l}$, $l\leq k, r\leq d, i\leq n$
such that conditions (i)-(v) of Lemma \ref{lemma1} hold (for $X_i^{(r)}$ and  $Y^{(r)}_{i,l}$ instead of $X$ and $Y_l$).
Let $H^{(r)}_{i}(x):=-\ln \Pro \left(Y_{i,l}^{(r)} \geq x \right)$ (observe that this function does not depend on $l$).

Let us start with the following Proposition.

\begin{fak}\label{fak0}
For any $p\geq 1$, 
\begin{align*}
\frac{1}{C(k,d)} \lv \sum_\sumator \ai \prod_{r=1}^d X^{(r)}_{i_r} \rv_p 
&\leq  \lv \sum \ai \prod_{r=1}^d\prod_{l=1}^k Y^{(r)}_{{i_r},l} \rv_p
\\
&\leq C(k,d) \lv \sum_\sumator \ai \prod_{r=1}^d X^{(r)}_{i_r} \rv_p.
\end{align*}
\end{fak}

\begin{proof}
Lemma \ref{lemma1} \eqref{ii} yields
\begin{align}
&\lv \sum_\sumator \ai \prod_{r=1}^d X^{(r)}_{i_r} \rv_p 
\nonumber
\\
&\leq C^d \lv \sum_\sumator \ai \prod_{r=1}^d \prod_{l=1}^k\left(Y^{(r)}_{{i_r},l}+t_0 \right) \rv_p  
\nonumber\\
&\leq C^d \sum_{\substack{\eps_r \in \{0,1\} \\ r=1,\ldots,d} } 
\lv \sum_\sumator \ai \prod_{r=1}^d\prod_{l=1}^k \left( (Y^{(r)}_{{i_r},l})^{\eps_r}t_0^{1-\eps_r}\right)\rv_p.
\label{fak01}
\end{align}

We have $\E Y_{i,l}^{(r)}\geq \frac{1}{C}$, so by Jensen's inequality we get for any $\eps\in\{0,1\}^d$,
\begin{align}
\lv \sum \ai \prod_{r=1}^d\prod_{l=1}^k Y^{(r)}_{{i_r},l} \rv_p
&\geq \lv \sum_\sumator \ai \prod_{r=1}^d\prod_{l=1}^k (Y^{(r)}_{{i_r},l})^{\eps_r}((\E Y^{(r)}_{{i_r},l}))^{1-\eps_r}\rv_p
\nonumber
\\
&\geq \frac{1}{C^{kd}}
\lv \sum_\sumator \ai \prod_{r=1}^d\prod_{l=1}^k \left((Y^{(r)}_{{i_r},l})^{\eps_r}(t_0)^{1-\eps_r} \right) \rv_p.
\label{fak02}
\end{align}

The lower estimate in Proposition \ref{fak0} follows by \eqref{fak01} and \eqref{fak02}. The proof of the upper bound
is analogous.
\end{proof}

So to prove Theorem \ref{th1} we need to estimate $\lv \sum \ai \pp Y_{i_r,l}^{(r)} \rv_p$. To this end we will apply
the following result of Lata{\l}a and {\L}ochowski.

\begin{twie}[{\cite[Theorem 2.1]{2}}]
\label{loch}
Let $\{Z^{(1)}_{i}\},\ldots, \{Z^{(d)}_{i}\}$ be independent nonnegative r.v's with logarithmically concave tails and 
$M^{(r)}_{i} (t)=-\ln \left( \Pro \left( Z^{(r)}_{i} \geq t\right) \right)$. 
Assume that $1=\inf \{t>0: M^{(r)}_{i} (t) \geq 1 \}$. Then 
\begin{align*}
\frac{1}{C} &\sup \left\{\sum_\sumator \ai \prod_{r=1}^d (1+b^{(r)}_{i_r}) \ | \ \left( b^{(r)}_i \right) \in T^{(r)}_{p} \right\} 
\\
&\leq \lv \sum_\sumator \ai Z^{(1)}_{i_1}\ldots Z^{(d)}_{i_d} \rv_p \\
&\leq C \sup \left\{\sum_\sumator \ai \prod_{r=1}^d (1+b^{(r)}_{i_r}) \ | \ \left(b^{(r)}_i\right) \in T^{(r)}_{p} \right\},
\end{align*}
where $T^{(r)}_{p}=\left\{b\in \R_{+}^n: \sum_{i=1}^n M^{(r)}_i(b_i) \leq p \right\}$.
\end{twie}

To use the above result we need to normalize variables $Y_{i,l}^{(r)}$. 
Let 
$$
t^{(r)}_i=\inf \{t>0\ : \ H^{(r)}_i(t) \geq 1 \},\ \ r\leq d,\ i\leq k.
$$
Lemma \ref{lemma1} \eqref{v} gives  
$t^{(r)}_i \leq e \E Y^{(r)}_{i,l}\leq e C$ and by \eqref{kluwlas} $t^{(r)}_i \geq t_0>0$, thus 
\begin{equation}
\label{ti}
\frac{1}{C(k,d)} \leq t^{(r)}_{i} \leq C(k,d).
\end{equation}

Theorem applied to variables $Y_{i,l}^{(r)}=Y_{i,l}^{(r)}/t_i^{(r)}$ together with \eqref{ti} gives
\begin{align*}
&\lv \sum_\sumator \ai \pp Y_{i_r,l}^{(r)} \rv_p 
\\
&\sim_{k,d} \nawbb  
\end{align*}
where
$$D^{(r)}_{k,p} = \left\{ \left(v_{i,l}\right)_{i\leq n; l \leq k} \in \R^{nk} \ : \ 
 \sum_{i=1}^n H^{(r)}_{i}(v_{i,l}) \leq p \mbox{ for all }l\leq k \right\}.$$
Lemma \ref{lemma1} \eqref{iv} yields
\begin{align*}
&\nawbb 
\\
&\sim_{k,d} \nawb
\\
&\phantom{aaaa}=:\normb,
\end{align*}
where
$$B^{(r)}_{k,p} = \left\{ \left(v_{i,l}\right)_{i\leq n; l \leq k} \in \R^{nk} \ : \ 
\sum_{i=1}^n N^{(r)}_{i}(v^k_{i,l}) \leq p \mbox{ for all }l\leq k \right\}.$$

To finish the proof of Theorem \ref{th1} we need to show that 
\begin{equation} \label{koniec}
\normb \sim \norma.
\end{equation} 

First we will show this holds for $d=1$, that is
\begin{align}\label{d1}
\sup\{\sum b_i \ki \  | \  \sum N_i(a_{i,l}^k)\leq p \mbox{ for all }l\leq k\} \nonumber \\
\sim_k  \sup\{\sum b_i (1+w_i) \  |  \ \sum N_i(w_i) \leq p\}.
\end{align} 

We have
\begin{align*}
&\sup\{\sum b_i \ki \ | \   \sum N_i(a_{i,l}^k)\leq p\mbox{ for all }l\leq k\} \\
&\leq \sum_{\substack{\eps_l \in \{0,1\}\\ l=1 \ldots k}}
\sup \{ \sum b_i \prod_{l=1}^k \aaa^{\eps_l}\ | \  \ \sum N_i(a_{i,l}^k)\leq p \mbox{ for all }l\leq k\}. 
\end{align*}
So to establish the upper bound in \eqref{d1} it is enough to prove 
\begin{align*}
&\sup \{ \sum b_i \prod_{l=1}^k \aaa^{\eps_l} \ | \  \ \sum N_i(a_{i,l}^k)\leq p\mbox{ for all }l\leq k\} \nonumber \\
&\leq C(k) \sup\{\sum b_i (1+w_i) \ |  \ \sum N_i(w_i)\leq p\}
\end{align*}
or equivalently (after permuting indexes) that for any $0\leq k_0 \leq k$,
\begin{align}\label{loc4}
&\sup \{\sum b_i \prod_{l=1}^{k_0} \aaa \ | \  \sum N_i(a_{i,l}^k)\leq p \mbox{ for all }l\leq k_0\} \nonumber \\
&\leq C(k) \sup\{\sum b_i (1+w_i) \ |  \ \sum N_i(w_i)\leq p\}.
\end{align}

Let us fix sequences $(a_{i,l})$ such that $\sum N_i(a_{i,l}^k)\leq p$ for all $l\leq k_0$. Let
$C$ be a constant from Lemma \ref{lemma0}, define
$$w_i=
\begin{cases}
\frac{\prod_{l=1}^{k_0} \aaa  }{C k^{k}} &\textrm{ if } \prod_{l=1}^{k_0} \aaa>2Ck^k\\ 
0 &\textrm{ otherwise}
\end{cases}
$$
For such $w_i$ we have $\prod_{l=1}^{k_0} \aaa \leq 2C k^k (1+w_i)$, so to establish \eqref{loc4} it is enough to check that
$\sum N_i(w_i) \leq p$. Lemma \ref{lemma0} however yields
\begin{align*}
\sum_i N_i(w_i) &\leq \kk \sum_i N_i(Ck^k\cdot w_i)  
\leq \kk \sum_{i\colon w_i\neq 0} N_i \left(\max\{a_{i,1},\ldots,a_{i,k_0} \}^{k_0} \right)\\
&\leq \kk \sum_i N_i \left(\max\{a_{i,1},\ldots,a_{i,k_0} \}^{k} \right)\leq \kk \sum_i \sum_{l=1}^{k_0} N_i(a^k_{i,l}) \leq p.
\end{align*}
where the third inequality comes from the observation that $w_i\neq 0$ implies \\ $\max\{a_{i,1},\ldots,a_{i,k_0} \}\geq 1$.

To show the lower bound in \eqref{d1} we fix $w_i \in B_p$, choose $a_{i,1}=a_{i,2}=\ldots=a_{i,k}=w_i^\kk$
and observe that
$$\sum b_i \ki =\sum b_i \left(1+w_i^\kk\right)^k \geq \sum b_i (1+w_i).$$

We showed that \eqref{d1} holds. Now we prove \eqref{koniec} for any $d$. We have
\begin{align} 
&\normb=\nawbc \nonumber \\
&=\sup \left\{ \sup\left\{\sum_{i_1,\ldots,i_d} a_{i_1,...,i_d} \prod_{r=1}^d \prod_{l=1}^k \left(1+  v^{(r)}_{i_r,l} \right)\ | 
\ (v^{(d)}_{i,l}) \in B^{(d)}_{k,p} \right\}   \ | \ \forall_{r \leq d-1} (v^{(r)}_{i,j}) \in B^{(r)}_{k,p}\right\} \nonumber \\ 
&\sim_k \sup\left\{\sum_{i_1,\ldots,i_d} a_{i_1,...,i_d} \prod_{r=1}^{d-1} \prod_{l=1}^k \left(1+  v^{(r)}_{i_r,l}\right) 
\left(1+w^{(d)}_{i_d}  \right)\ |\ (w^{(d)}) \in B^{(d)}_p,\forall_{r \leq d-1} v^{(r)}_{i,j} \in B^{(r)}_{k,p} \right\}, 
\end{align}
where the last equivalence follows by \eqref{d1}. Iterating the above procedure $d$ times we obtain \eqref{koniec}.

\begin{rem}
Deeper analisys of the proof shows that constant $C$ from  Theorem \ref{th1} is less then $(C')^{k^3 d} k^{6k^2d} d^{kd}$ for $C'$ a universal constant.

\end{rem}

\textbf{Acknowledgements:} I would like to express my gratitude to my wife, for supporting me. I would also like to thank  prof. R. Lata{\l{}}a for a significant improvement of the paper.

\noindent
Rafa{\l} Meller\\
Institute of Mathematics\\
University of Warsaw\\
Banacha 2, 02-097 Warszawa, Poland\\
{\tt r.meller@mimuw.edu.pl}

\end{document}